\documentclass[12pt,a4paper,reqno]{amsart}

\usepackage[latin1]{inputenc}
\usepackage{amssymb,amsmath,amstext,amsthm,amsfonts}
\usepackage{graphicx}
\usepackage{multicol}
\usepackage{subfigure}
\usepackage{hyperref}

\usepackage{a4wide}
\usepackage{xcolor}

\newcommand{\R}{\mathbb{R}}

\newcommand{\tent}{T}

\newcommand{\qand}{\quad\text{and}\quad}

\def\inte{\operatorname{int}}
\def\diam{\operatorname{diam}}
\def\dist{\operatorname{dist}}
\def\var{\operatorname{var}}

\newtheorem{maintheorem}{Theorem}

\newcommand{\cmt}{\begin{maintheorem}}
\newcommand{\fmt}{\end{maintheorem}}

\newtheorem{maincorollary}[maintheorem]{Corollary}

\newcommand{\cmc}{\begin{maincorollary}}
\newcommand{\fmc}{\end{maincorollary}}

\newtheorem{lemma}{Lemma}[section]

\newtheorem{proposition}[lemma]{Proposition}

\theoremstyle{remark}
\newtheorem{remark}[lemma]{Remark}

\thanks{JFA is partially supported by     CMUP (UID/MAT/00144/2013), PTDC/MAT-CAL/3884/2014 and FAPESP/19805/2014, which are funded by FCT (Portugal) with national (MEC) and European structural funds through the programs COMPTE and FEDER, under the partnership agreement PT2020. AP is partially supported by projects MINECO-15-MTM2014-56953-P and MTM2017-87697-P}

\begin{document}
\title[Entropy formula and continuity of  entropy]{Entropy formula and continuity of   entropy for      piecewise expanding  maps}
\date{}

\author[J. F. Alves]{Jos\'{e} F. Alves}
\address{Jos\'{e} F. Alves\\ Centro de Matem\'{a}tica da Universidade do Porto\\ Rua do Campo Alegre 687\\ 4169-007 Porto\\ Portugal}
\email{jfalves@fc.up.pt} \urladdr{http://www.fc.up.pt/cmup/jfalves}

\author[A. Pumari\~no]{Antonio Pumari\~no}
\address{Antonio Pumari\~no\\ Departamento de Matem\'aticas, Facultad de Ciencias de la Universidad de Oviedo, Federico Garc\'{\i}a Lorca, 18, 33007 Oviedo, Spain.}
\email{apv@uniovi.es} 

\maketitle




\begin{abstract}
We consider some classes of  piecewise expanding maps in    finite dimensional spaces having  invariant probability measures which are absolutely continuous with respect to Lebesgue measure. We derive an entropy formula  for such  measures and, using this entropy formula, in some parametrized families we present sufficient conditions for the continuity  of that  entropy with respect to the parameter. We apply our results to  a classical one-dimensional family of tent maps and a family of two-dimensional maps  which  arises as the  limit of return maps 
 when a homoclinic tangency is unfolded by a family of three dimensional diffeomorphisms.
\end{abstract}


\tableofcontents

\section{Introduction}

Be it good or  bad, in the mathematical theory of Dynamical Systems one can easily find many examples of systems  with simple evolution laws whose dynamics is very complex and hard  to predict in deterministic terms. 
Just to mention a few, in this direction we refer the one-dimensional quadratic maps, the two-dimensional Hénon quadratic diffeomorphisms, or the system of Lorenz quadratic differential equations in the three-dimensional Euclidean space. Though simple in formulation, all these systems have very complicate dynamical behavior and, in the last decades, have  motivated   the appearance of relevant mathematical results   in several directions. 

Among the many important contributions to the theory, there are several results both on the existence of \emph{Sinai-Ruelle-Bowen (SRB)  measures}, i.e. ergodic  invariant  probability measures whose conditionals on local unstable manifolds are absolutely continuous with respect to the conditionals of Lebesgue measure,
and on the continuous dependence of these SRB measures (or their entropies)  with respect to the underlying dynamics; see   \cite{A00, ACF10,ACF10a,AOT06, AS14,APP09,BC85,BC91,BY92,BY93a,F05}. In all these situations, the SRB measures are known to be physically relevant, in the sense that they describe the statistics of many initial states of the systems, frequently almost all  initial states with respect to the Lebesgue (volume) measure on the ambient space.

Mostly motivated by   the family of two-dimensional tent maps introduced in~\cite{PRT14} (see Subsection~\ref{se.tents} below), in this work we present  some  general results on the continuity of the entropy of ergodic absolutely continuous invariant probability measures for some classes of piecewise expanding maps in any finite dimension, possibly with infinitely many domains of smoothness. As a main application of these results we shall consider the family of  two-dimensional tent maps  considered   in \cite{PRT14}. This family   is particularly interesting because it is related to the   limit return maps arising when a homoclinic tangency is unfolded by a family of three dimensional diffeomorphisms;
see \cite{PRT14} and \cite{T01a} for details.
After having proved  in~\cite{PRT15} the existence of ergodic absolutely continuous invariant probability measures for these tents maps, and   in~\cite{APV17}  the continuity of such measures, it is then natural to ask under which conditions  the  metric  entropy  with respect to those measures  depends  continuously on the dynamics.

Our strategy  to prove the continuity of  the metric entropy  is heavily based on the validity of an \emph{entropy formula} for invariant measures, particularly when the measure is absolutely continuous with respect to the reference Lebesgue measure. For the case of smooth diffeomorphisms of a Riemannian manifold, Ruelle  established in~\cite{R78} that the entropy of any invariant probability measure is bounded by the  integral of the sum of the positive Lyapunov exponents (counted with multiplicity) with respect to that measure. The reverse inequality  has been obtained  in~\cite{P77}  by Pesin for  the case that the invariant probability measure is absolutely continuous with respect to the Lebesgue measure. Natural versions for non-invertible smooth maps have been drawn in \cite{QXZ09}. Extensions of the results of Ruelle and Pesin for  the the class of maps with infinite derivative introduced  in   \cite{KSL86}  were obtained in  \cite{LS82}. Conversely, the existence of an entropy formula for can as well be used to prove that an invariant probability  measure has conditional measures on local unstable manifolds which are absolutely continuous with respect to the conditionals of Lebesgue measure on those manifolds; see  \cite{L81,L84,LY85}.

In the context of non-invertible maps, we are naturally lead to consider the case where all Lyapunov exponents are positive and the sum of Lyapunov exponents coincides with the Jacobian of the map.
Surprisingly, we did not find in the literature any result that could be directly used  to assure that the  tent maps  in~\cite{PRT14}  satisfy the entropy formula. 
Actually, to the best of our knowledge, not much is known on the existence of entropy formulas  for piecewise smooth maps, specially in dimension greater than one. For one-dimensional dynamical systems   see e.g.  \cite{BJ12}, \cite{D14}, \cite{K89} or  \cite{L81}.   In higher dimensions,  \cite{DKU90} is the closest to our setting  that we could find   in the literature. However, a technical  assumption in \cite[Theorem~1]{DKU90} that we could not be verify in our tent maps, prevented us from applying that result; see condition  \eqref{eq.tecnica} below. 
Let us refer that in the Markovian case of piecewise expanding maps with full branches (which is not the case of  our tent maps), the situation is completely different: not only the assumptions of \cite[Theorem~1]{DKU90} can be easily verified, but also a direct approach as   in~\cite[Section~4]{AOT06} can be implemented to obtain an entropy formula. 

Let us point out that, though  our initial motivation for this work is the aforementioned two-dimensional family of tent maps, our results on the existence of an entropy formula and continuity of the entropy hold    for much more  general  families of piecewise expanding maps with infinitely many domains of smoothness in any finite dimension.

\subsection{Piecewise expansion and bounded distortion}\label{se.pebd}


Let  $\Omega$ be a compact subset of~$\R^d$, for some $d\ge 1$. Consider   $m$ the Lebesgue (or volume) measure on $\Omega$ and, for each $1\le p\le\infty$, the space $L^p(m)$  endowed with its usual norm $\|\quad\|_p$.  Throughout this paper, absolute continuity will be always meant with respect to the Lebesgue measure $m$. 

Assume that $\phi: \Omega\to \Omega$ is a map for which there is a (Lebesgue mod 0) countable partition
$\mathcal R_\phi$ of $\Omega$ such that each $R\in \mathcal R_\phi$ is a
closed domain with piecewise $C^1$ boundary of finite
$(d-1)$-dimensional measure. Assume also that 
$\phi_R:=\phi|_{R}$ is a $C^1$ bijection from  $\inte(R)$, the interior  of~$R$, onto its image, with a $C^1$ extension to~$R$. We say that $\phi$ is a  {\it piecewise
expanding map} if
\begin{itemize}
\item[(P$_1$)] there is $0<\sigma<1$ such that for all $R\in\mathcal R_\phi$ and all $x\in\inte(\phi(R))$
$$\| D\phi_R^{-1}(x)\|
\le \sigma.$$ 
 \end{itemize}
Consider the Jacobian function
 $$J_{\phi}=|\det (D\phi)|,$$
naturally defined on the (full Lebesgue measure)  subset of points in~$\Omega$ where $\phi$ is differentiable. We say that  $\phi$ has {\it
bounded distortion} if
\begin{itemize}
 \item[(P$_2$)] there
is $\Delta\ge 0$ such that for all $R\in\mathcal R$ and all $x,y\in\inte(R)$ 
$$
\log \frac{J_{\phi}(x)}{J_{\phi}(y)}\le \Delta\,\|\phi(x)-\phi(y)\|.
$$
 \end{itemize}

Sufficient conditions for the existence of an  absolutely continuous invariant probability measure    for a piecewise expanding map with bounded distortion are given in~\cite{GB89} for finitely many  domains of smoothness,  and in~\cite[Section 5]{A00} for infinitely many domains. In the latter case, these conditions correspond to property (P$_3$) and~\eqref{betasigma} below on the images of the smoothness domains. Similar conditions  are imposed  on the domains themselves, in the finitely many domains case considered in~\cite{GB89}.

\subsection{Entropy formula}\label{se.entropf}
One of the main goals of this work is to obtain an entropy formula for  an absolutely continuous invariant probability measure of a $C^1$ piecewise expanding map $\phi: \Omega\to \Omega$. Actually, one of the inequalities will be obtained for invariant   measures not necessarily absolutely continuous. 

We start by recalling the notion of entropy of $\phi$  with respect to an invariant  measure $\mu$. The \emph{entropy of a  partition} $\mathcal P$ of $\Omega$ with respect to a measure $\mu$ is given by 
$$
\displaystyle H_\mu(\mathcal P)=-\sum_{P\in\mathcal P}^{\infty}\mu\left(P\right)\log  \mu\left(P\right),
 $$
 the \emph{entropy of $\phi$ with respect to $\mu$ and a partition $\mathcal P$} is given by
 $$h_\mu(\phi,\mathcal P)=\lim_{n\to\infty}\frac1nH_\mu(\mathcal P^n),$$
 where for each $n$ 
  \begin{equation}\label{parts}
 \mathcal P^n=\bigvee_{j=0}^{n-1}\phi^{-j}(\mathcal P).
 \end{equation}
 Finally,  the \emph{entropy of $\phi$ with respect to $\mu$} is given by
  $$h_\mu(\phi)=\sup_{\mathcal P}h_\mu(\phi,\mathcal P).$$

We say that the partition $\mathcal R_\phi$ into the smoothness domains of $\phi:\Omega\to\Omega$ is \emph{quasi-Markovian} with respect to a measure $\mu$, if  there exists $\eta>0$ such that for $\mu$ almost every $x\in \Omega$,  there are infinitely many values of $n\in\mathbb N$ for which
\begin{equation*}
m(\phi^n(R^n(x)))\ge\eta,
\end{equation*}
where $R^n(x)$ stands for  the element in $\mathcal R^n_\phi$  containing~$x\in \Omega$. In this definition we are implicitly assuming that    $\mathcal R_\phi^n$ is a $\mu$ mod 0 partition of $\Omega$, for all $n\ge1$.

Our first result relates the entropy with the integral of the Jacobian for a piecewise expanding map, and it holds for invariant probability measures in general, not necessarily absolutely continuous.

\begin{maintheorem}\label{ruelle}
Let $\phi$ be a  $C^1$
 piecewise expanding map with bounded distortion. If  $\mu$ is a $\phi$-invariant probability measure such that $H_\mu(\mathcal R_\phi)<\infty$  and $\mathcal R_\phi$ is quasi-Markovian with respect to  $\mu$,  then
$$h_{\mu}(\phi)\le \int \log J_{\phi} d\mu.$$
\end{maintheorem}

We also obtain the reverse inequality  if the invariant probability measure is absolutely continuous with respect to Lebesgue measure. 

 \begin{maintheorem}\label{pesin}
  Let   $ \phi $ be a  $C^1$
piecewise expanding map  with bounded distortion. If   $\mathcal R_\phi$ is quasi-Markovian with respect to  an absolutely continuous invariant probability measure $\mu$ such that $H_\mu(\mathcal R_\phi)<\infty$,   then  $$h_{\mu}(\phi)=\int \log J_{\phi} d\mu.$$ \end{maintheorem}
 
We shall refer to the conclusion of Theorem~\ref{pesin} as the \emph{entropy formula} for $\mu$.

The same conclusion of~Theorem~\ref{pesin}  has been drawn in  \cite[Proposition~4.1]{AOT06}, under the stronger Markovian assumption of  maps  with  \emph{full branches}. Also in  \cite[Theorem~1]{DKU90}, in the more general setting of  a measurable transformation and  a conformal reference measure, replacing our  quasi-Markovian condition by 
\begin{equation}\label{eq.tecnica}
\sup_{n\geq 1}\left| \int \log m(\phi^n (R^n (x))) d \mu(x)    \right|  < \infty, 
\end{equation}
where $\mu$ is absolutely continuous with respect to the reference measure $m$; see  (2.4) in~\cite{DKU90}. Even though our quasi-Markovian condition introduced above has the same flavor of~\eqref{eq.tecnica},  in practice our condition is easier to deal with. Actually, we were able to check  it  for the family of tent maps that we introduce  in Subsection~\ref{se.tents} below and we were not able to check condition~\eqref{eq.tecnica} for that family.


Our next goal is to establish some useful criterium for obtaining the quasi-Markovian property for the partition of a piecewise expanding map with respect to an absolutely continuous invariant probability measure. We define the \emph{singular set} of a piecewise expanding map $\phi$ as
$$
\mathcal S_\phi =\overline{\bigcup_{R\in \mathcal R_\phi}\partial R},
$$
 where $\partial$ stands for the boundary and  bar  for the closure of a set. Notice that when the partition $\mathcal R_\phi$ is finite the singular  set $\mathcal S_\phi$ is a finite union of $(d-1)$-dimensional submanifolds of $\mathbb R^d$. We say that a  piecewise expanding maps  $\phi$ \emph{behaves as a power of the distance} close to   $\mathcal S_\phi$ if 
 there exist constants $B,\beta>0$ such that
\begin{itemize}
\item[(S1)] 
$\displaystyle{
  {\|D\phi(x) \|} 
\le \frac{B}{\dist(x,\mathcal S_\phi )^{\beta}}}$;

\item[(S2)] 
$\displaystyle{
\log\frac{\|D\phi(x)^{-1}\|}{ \|D\phi(y)^{-1}\|}
\le  \frac{B}{\dist(x,\mathcal S_\phi )^{\beta}}\dist(x,y)}$;

\end{itemize}
for every   $x, y \in M\setminus\mathcal S_\phi $
with $\dist(x,y) < \dist(x,\mathcal S_\phi )/2$. 

In Proposition~\ref{pr.casiM} we establish a criterium for the quasi-Markovian property of the partition associated to a piecewise expanding map behaving as a power of the distance close to the singular set. Using that criterium we easily obtain the next result as a consequence of Theorem~\ref{pesin}. This will be particularly useful for establishing the entropy formula of the family of tent maps in  Subsection~\ref{se.tents}.

  \begin{maintheorem}\label{pesin2}
  Let   $ \phi $ be a  $C^1$
piecewise expanding map  with bounded distortion   for which (S1)-(S2) hold, and  let $\mu$ be an ergodic absolutely continuous invariant probability measure  for $\phi$ such that   $H_\mu(\mathcal R_\phi)<\infty$. If $\log\dist(\cdot,\mathcal S_\phi)\in L^p(m)$ and $d\mu/dm\in L^q(m)$ with $1\le p,q\le\infty$ and $1/p+1/q=1$, 
   then  $$h_{\mu}(\phi)=\int \log J_{\phi} d\mu.$$ \end{maintheorem}

Theorems~\ref{ruelle} and~\ref{pesin} are proved in Section~\ref{se.entropy}, Theorem~\ref{pesin2} is proved in Section~\ref{se.quasi-markovian}. 

\subsection{Continuity of entropy} Now we  consider families of piecewise expanding maps. Let~$I$ be a metric space and   $(\phi_t)_{t\in I}$ be  a family of $C^1$
piecewise expanding maps $\phi_t:\Omega\to \Omega$, where $\Omega$ is a compact subset of $\R^d$, for some $d\ge 1$. For simplicity, for each  $t\in I$ denote by~$\mathcal R_t$ the partition of $\Omega$ associated to the piecewise expanding map~$\phi_t$.
%
%
%
%
%
Our next result gives sufficient conditions for the continuity of  the entropy of the absolutely continuous invariant probability measure.

\begin{maintheorem}\label{entropy} Let  $(\phi_t)_{t\in I}$ be a family of $C^1$
piecewise expanding maps with bounded distortion such that each 
$\phi_t$ has   an absolutely continuous invariant probability measure  $\mu_t$  for which $H_{\mu_t}(\mathcal R_t)<\infty$ and the entropy formula holds.
Assume that there are  $1< p,q\le\infty$ with $1/p+1/q<1$ such that
\begin{enumerate}
\item $d\mu_t/dm$ is uniformly bounded in $L^p(m)$ and depends continuously on $t\in I$ in  $L^1(m)$;
\item $\log J_{\phi_t}\in L^q(m)$ and $\log J_{\phi_t}$ depends continuously on $t\in I$ in  $L^1(m)$.
\end{enumerate}
 Then $h_{\mu_t}(\phi_t)$ depends continuously on $t\in I$.
 \end{maintheorem}
 
The proof of this theorem will be given in  Section~\ref{se.continuity} and uses the entropy formula of Theorem~\ref{pesin}.  In the next subsections we introduce more particular settings where the previous  results above can be applied.
 
 \subsection{Maps with large branches} \label{se.mlb}
A special class of piecewise expanding maps $\phi:\Omega\to \Omega$ with bounded distortion for which an absolutely continuous invariant probability measure always exists has been introduced in \cite{A00}, namely   maps that satisfy the following additional condition on the images of the smoothness domains. 
We say that $\phi$ has \emph{large branches} 
 if 
\begin{enumerate}
 \item[(P$_3$)]   there are constants $\alpha,\beta >0$ and for each $R\in\mathcal R_\phi$ there is a $C^1$
unitary vector field~$X$ in $\partial \phi(R)$\footnote{ At the   points $x\in\partial \phi(R)$
where $\partial \phi(R)$ is not smooth the vector $X(x)$ is a
common $C^1$ extension of $X$ restricted to each
$(d-1)$-dimensional smooth component of $\partial \phi(R)$ having
$x$ in its boundary. The tangent space
at any such point   is the union of the tangent
spaces to the  $(d-1)$-dimensional smooth components that point 
belongs to.} such that:
\begin{enumerate}
\item[(a)]   the line segments joining each $x\in\partial \phi(R)$ to
$x+\alpha X(x)$ are pairwise disjoint contained in $\phi(R)$, and their union form a neighborhood of $\partial \phi(R)$ in $\phi(R)$;
\item[(b)] for every $x\in\partial \phi(R)$
and $v\in T_x\partial \phi(R)\setminus\{0\}$ the angle $\theta(x,v)$ between $v$ and $X(x)$ 
satisfies $\left|\sin\theta(x,v)\right|\geq\beta.$
\end{enumerate}
 \end{enumerate}
 
 In the one-dimensional case, condition (P$_3$)(a)  is obviously satisfied when  the elements in~$\mathcal R_\phi$ are intervals whose images   have sizes uniformly bounded away from zero.

 \begin{remark} \label{re.beta}
 Condition  (P$_3$)(b) makes no sense in dimension one, but in practice we can assume  the optimal value~$\beta=1$; see \cite[Remark~3.2]{APV17}.
 \end{remark}

From   \cite[Theorem 5.2]{A00} we know that if $\phi:\Omega\to \Omega$ is a $C^2$ piecewise expanding  map  for which conditions (P$_1$)-(P$_3$)\footnote{ Our bounded distortion condition (P$_2$)  is slightly different from  the one used  in \cite{A00}. However, in Lemma~\ref{le.distortion}  we show that  (P$_2$) implies the bounded distortion condition in \cite{A00}, exactly   with the same constant $\Delta>0$.}  hold with 
\begin{equation}\label{betasigma}\tag{$\ast$}
\sigma\left(1+1/{\beta}\right)<1,
\end{equation} then $\phi$ has
a finite number of absolutely continuous invariant probability measures.  Under these assumptions,  we also have that  the density of any such measure belongs to  $BV(\Omega)$, the space of bounded variation functions; see \cite[Corollary 3.4]{APV17}.   Assuming $ \Omega\subset \mathbb R^d$,   from   Sobolev Inequality  we deduce that   $BV(\Omega)$ is contained in $ L^{d/(d-1)}(m)$; see e.g. \cite[Theorem~1.28]{G84}. 
Since $1/d+(d-1)/d=1$, using  Theorem~\ref{pesin2} we easily get the following result.

  \begin{maincorollary}\label{pesin3}
  Let   $ \phi :\Omega\to\Omega$, with $\Omega\subset \mathbb R^d$, be a  $C^1$
piecewise expanding map  with bounded distortion and large branches for which (S1)-(S2) and \eqref{betasigma} hold. If  $\log\dist(\cdot,\mathcal S_\phi)\in L^d(m)$ and   $\mu$ is an ergodic absolutely continuous invariant probability measure for $\phi$ such that  $H_\mu(\mathcal R_\phi)<\infty$, 
   then  $$h_{\mu}(\phi)=\int \log J_{\phi} d\mu.$$ \end{maincorollary}

%
In addition,   
\cite[Theorem A]{APV17}
asserts that if   $(\phi_t)_{t\in I}$ is a family of $C^1$
piecewise expanding maps $\phi_t:\Omega\to \Omega$ satisfying   properties  (1) and (2) of Theorem~\ref{entropy2} below,
then
$(\phi_t)_{t\in I}$ is  \emph{statistically stable}: for any  sequence $(t_n)_n$ in $I$ converging to $t_0\in I$ and any   sequence of ergodic absolutely continuous $\phi_{t_n}$-invariant probability measures $(\mu_{t_n})_n$, any accumulation point of   the    densities $d\mu_{t_n}/dm$ must  converge in the $L^1$-norm to the density of an absolutely continuous $\phi_{t_0}$-invariant probability measure. Obviously, when each $\phi_t$ has a unique  (hence ergodic) absolutely continuous invariant probability measure~$\mu_t$, then  statistical stability means continuity of   $d\mu_t/dm$   in the $L^1$-norm with $t\in I$.  In the next result we give sufficient conditions for the continuity of the  entropy of the absolutely continuous invariant probability measure in the setting of piecewise expanding maps with large branches. 

\begin{maintheorem}\label{entropy2} 
Let  $(\phi_t)_{t\in I}$ be a family of $C^2$
piecewise expanding maps  with bounded distortion and large branches such that  each $\phi_t$ has a unique absolutely continuous invariant probability measure $\mu_t$ for which $H_{\mu_t}(\mathcal R_t)<\infty$ and the entropy formula holds.
Assume that 
\begin{enumerate}
\item there are
$0<\lambda<1$ and $K>0$ such that for each $t\in I$
 $$\displaystyle\sigma_t\left(1+\frac1{\beta_t}\right)\le \lambda\qand \displaystyle \Delta_t+\frac{1}{\alpha_t\beta_t}+\frac{\Delta_t}{\beta_t}  \le K,$$
 where $\sigma_t, \Delta_t,\alpha_t,\beta_t$ are constants such that (P$_1$)-(P$_3$) hold for $\phi_t$; 
 \item $f\circ\phi_t$ depends continuously on $t\in I$ in   $L^d(m)$, for each  continuous $f:\Omega\to\mathbb R$;
\item $\log J_{\phi_t}\in L^q(m)$ for some $q>d$  and $\log J_{\phi_t}$ depends continuously on $t\in I$ in   $L^1(m)$.
\end{enumerate}
 Then $h_{\mu_t}(\phi_t)$ depends continuously on $t\in I$.
 \end{maintheorem}
 
 The $C^2$ differentiability   assumed  in this last result  is due to the bounded distortion condition considered in~\cite{A00}.
The proof of Theorem~\ref{entropy2}  will be given in Section~\ref{se.large}. 

\subsection{Tent maps} \label{se.tents}
Here we consider some special families of piecewise expanding maps, the first one the family of two-dimensional tent maps introduced in~\cite{PRT14}, and
 the second one an  analogous family of interval maps.
The main results of this section will be obtained as 
an application of the previous results for piecewise expanding maps.

\subsubsection{Two-dimensional  maps}

 Consider the triangles
\begin{equation}\label{algo}
\mathcal{T}_0=\{(x,y):0 \leq x \leq 1, \ 0 \leq y \leq x \}
\end{equation}
and
\begin{equation*}
\mathcal{T}_1=\{(x,y):1 \leq x \leq 2, \ 0 \leq y \leq 2-x \}.
\end{equation*}
For each $0<t\leq 1$, define the map  $\tent_{t}:\mathcal{T}\to \mathcal{T}$ on the triangle $\mathcal{T}=\mathcal{T}_0\cup \mathcal{T}_1$ by
\begin{equation}\label{familyt2}
\tent_t(x,y)= \left\{
\begin{array}{ll}
(t(x+y),t(x-y)), & \mbox{if }    (x,y)\in \mathcal{T}_0;\\
(t(2-x+y),t(2-x-y)), & \mbox{if } (x,y)\in \mathcal{T}_1.%
\end{array}
\right.
\end{equation}
The  domains $\mathcal{T}_0$ and $\mathcal{T}_1$ are separated by the straight line segment
\begin{equation}\label{critical}
\mathcal C=\{(x,y)\in \mathcal{T}: x=1\}
\end{equation}
that we call the \emph{critical set} of $\tent_t$.

As shown in \cite{PT06}, the map $\tent_1$ displays the same properties of the well-known one-dimensional tent map defined for $x\in[0,2]$ as
 $x\mapsto 2-2|x-1|$. 
Among them,  the consecutive
pre-images $\{\tent_1^{-n}(\mathcal{C})\}_{n\in \mathbb{N}}$ of the critical line $\mathcal{C}$ define a sequence of partitions of~$\mathcal{T}$ whose diameter tends to zero as $n$ goes to infinity.
This enables us  to conjugate $\tent_1$ to a one sided shift with two symbols, from which  it easily follows that
$\tent_1$ is transitive in $\mathcal{T}$. Furthermore, the Lyapunov exponent of any point in $\mathcal{T}$ whose orbit 
does not hit the critical line is positive (actually, it coincides with $1/2\log{2}$) in every nonzero direction.
%
%
%
Finally, there is  a (unique) absolutely continuous invariant probability measure for $\tent_1$; see~\cite{PT06} for details. 

The results obtained in \cite{PT06} for  $t=1$ have been extended to a larger set of parameters. More precisely,  it was proved  in \cite{PRT15}  that
for each $t \in [\tau,1] $, with $ \tau=\frac{1}{\sqrt{2}}(\sqrt{2}+1)^\frac{1}{4}\approx 0.882,$ the map $\tent_t$  exhibits a \emph{strange attractor} $A_t\subset \mathcal{T}$, i.e.   $\tent_t$ is (strongly) transitive in $A_t$, the periodic orbits are dense in $A_t$, and  there exists a dense orbit in $A_t$ with two positive Lyapunov exponents. Furthermore,  $A_t$ supports a unique absolutely continuous invariant probability measure $\mu_t$.  As shown in \cite{APV17}, these measures $\mu_t$ depend continuously on  $t\in[\tau,1]$ in a strong sense: their densities vary continuously with the parameter in the norm of $L^1(m)$.
Here we show that the entropy with respect to the absolutely continuous invariant probability measure depends continuously on the parameter as well.

\begin{maintheorem}\label{main2} Each  $\tent_t$ has a unique absolutely continuous invariant probability measure $\mu_t$ depending continuously on $t\in[\tau,1]$. Moreover,  the entropy formula holds for $\mu_t$ and
 $h_{\mu_t}(\tent_t)$   depends continuously on $t\in[\tau,1]$.
\end{maintheorem}

The existence and uniqueness of the  absolutely continuous invariant probability measure~$\mu_t$  has already been proved in    \cite{PRT15}, and its continuous dependence (in a strong sense) on the parameter $t$  proved in \cite{APV17}.

\subsubsection{One-dimensional   maps} Though easier to deal with, but not following as an immediate  consequence of the results for the family of two-dimensional tent maps presented above, we can as well obtain similar conclusions for the one-dimensional family of tent maps $T_t:[0,2]\to[0,2]$, defined for $1<t\le 2$ and $x\in[0,2]$ as
\begin{equation}\label{familyt}
T_t(x)= \left\{
\begin{array}{ll}
tx, & \mbox{if }    0\le x\le 1;\\
t(2-x), & \mbox{if } 1\le x\le 2.%
\end{array}
\right.
\end{equation}

\begin{maintheorem}\label{main1} Each  $T_t$ has a unique absolutely continuous invariant probability measure $\mu_t$ depending continuously on $t\in(1,2]$. Moreover,  the entropy formula holds for $\mu_t$ and
 $h_{\mu_t}(T_t)$   depends continuously on $t\in(1,2]$.
\end{maintheorem}

The existence of an  absolutely continuous invariant probability measure for the maps in this family  follows as an easy consequence of   \cite[Theorem~1]{LY73}. We did not find in the literature any reference to the continuity of this measure or its entropy with the parameter.

\section{Entropy formula}\label{se.entropy}
In this section we   prove Theorem~\ref{ruelle} and Theorem~\ref{pesin}. Let $\phi$ be a  $C^1$
 piecewise expanding map with bounded distortion.
For simplicity, denote $\mathcal R_\phi$ by $\mathcal R$ and, for each $n\in\mathbb N$, let $\mathcal R^n$ be  as in~\eqref{parts}. 
Condition (P$_1$) implies  that $$\diam(\mathcal R^n)\to 0,\quad\text{as $n\to\infty$}.$$ 
It follows from Kolmogorov-Sinai Theorem that 
\begin{equation}\label{KS}
h_{\mu}(\phi)=h_{\mu}(\phi,\mathcal R)
\end{equation}
for any $\phi$-invariant measure $\mu$ such that $\mathcal R$ is a $\mu$ mod 0 partition of $\Omega$.
Moreover, assuming    $H_\mu(\mathcal R)<\infty$,     Shanon-McMillan-Breiman Theorem  gives that 
\begin{equation*}\label{eq.fracrn}
h_\mu(\mathcal R,x)=\lim_{n\to\infty}-\frac{1}{n} \log\mu(R^n(x))
\end{equation*}
is well defined for $\mu$ almost every $x\in \Omega$, and 
\begin{equation*}\label{SMB}
h_{\mu}(\phi,\mathcal R)=\int h_\mu(\mathcal R,x)d\mu.
\end{equation*}
Together with~\eqref{KS} this gives 
\begin{equation}\label{eq.entro}
h_\mu(\phi)=\int h_\mu(\mathcal R,x)d\mu.
\end{equation}
In the case that $\mu$ being ergodic, Shanon-McMillan-Breiman Theorem also gives that
\begin{equation}\label{eq.fracrn2}
h_\mu(\phi)=\lim_{n\to\infty}-\frac{1}{n} \log\mu(R^n(x))
\end{equation}
for $\mu$ almost every $x$.
Notice that \eqref{eq.entro} holds whenever $\mu$ is a   $\phi$-invariant probability measure with  $H_\mu(\mathcal R)<\infty$ and $\mathcal R$ is a $\mu$ mod 0 partition of $\Omega$. Moreover, \eqref{eq.fracrn2} holds when  $\mu$ is additionally ergodic.

Now we give a simple bounded distortion result that will be used in the proofs of both inequalities  that yield  the entropy formula.

\begin{lemma}\label{le-distor}
There is $C>0$ such that for all $R_n\in \mathcal R^n$ and all $x,y\in R_n$
$$\frac1C\le \frac{J_{\phi^n}(x)}{J_{\phi^n}(y)}\le C.$$
\end{lemma}
\begin{proof}
For all $x,y\in R_n$ we may write
\begin{align*}
\log\frac{J_{\phi^n}(x)}{J_{\phi^n}(y)}&=\sum_{j=0}^{n-1}\log\frac{J_{\phi}(\phi^{i}(x))}{J_{\phi}(\phi^{i}(y))}\\
&\le \sum_{j=0}^{n-1}\Delta\|\phi^j(x)-\phi^j(y)\|\\
&\le  \sum_{j=0}^{n-1}\Delta\sigma^{n-j}\|\phi^n(x)-\phi^n(y)\|,\quad\text{by (P$_1$)  }\\
&\le \sum_{j=0}^{n-1}\Delta\sigma^{n-j}\diam(\Omega).
\end{align*}
Taking exponentials and using the symmetry on the roles of $x$ and $y$ we easily finish the proof.
\end{proof}
Now we split the proofs of Theorems~\ref{ruelle} and~\ref{pesin} into the next two inequalities.

\subsection{First inequality}
In this subsection we complete the proof of Theorem~\ref{ruelle}, following some     ideas from the proof of~\cite[Theorem 1]{H91}.
Take 
   $\mu$ a $\phi$-invariant probability measure such that $H_\mu(\mathcal R_\phi)<\infty$ and assume that $\mathcal R_\phi$ is quasi-Markovian with respect to~$\mu$. 
   
   Consider first the case in which  $\mu$ is an ergodic probability measure. Assuming by contradiction that
the conclusion of Theorem~\ref{ruelle} does not hold, 
choose real numbers $\alpha,\beta$ such that
\begin{equation}\label{absurdo}
h_{\mu}(\phi)
> \alpha>\beta>\int \log J_{\phi}d\mu.
\end{equation}
It follows from~\eqref{eq.fracrn2}   and \eqref{absurdo} 
that there exists a set $E_1\subset \Omega$ with $\mu(E_1)=1$ such that for every $x\in E_1$ there exists $k_1=k_1(x)\in\mathbb N$ for which
\begin{equation}\label{muYn}
\mu\left(R^n(x)\right)\le e^{-\alpha n},\quad\forall n\ge k_1.
\end{equation}
Moreover, by Birkhoff's Ergodic Theorem we have
$$\frac1n\log J_{\phi^n}(x)=\frac1n\sum_{j=0}^{n-1} \log J_\phi(\phi^j(x))\to \int\log J_\phi d\mu,\quad\text{as $n\to\infty$,}
$$
which then implies that there exists a set $E_2\subset \Omega$ with $\mu(E_2)=1$ such that for every $x\in E_2$ there exists $k_2=k_2(x)\in\mathbb N$ for which
 \begin{equation}\label{equacaoJ}J_{\phi^n}(x)\le e^{\beta n},\quad \forall n\ge k_2.
 \end{equation}
 On the other hand, there exist $\eta>0$ and  a set $E_3\subset \Omega$ with $\mu(E_3)=1$ such that every $x\in E_3$ there are infinitely many values of $n\in\mathbb N$  for which
 \begin{equation}\label{equacao}
m(\phi^n(R^n(x)))\ge\eta.\end{equation}
Using Lemma~\ref{le-distor},  for each $n\in\mathbb N$ we have
\begin{align}
 m\left(\phi^n(R^n(x))\right)&=\int_{R^n(x)}J_{\phi^n}(y)dm(y)\nonumber \\
&=\int_{R^n(x)}\frac{J_{\phi^n}(y)}{J_{\phi^n}(x)}J_{\phi^n}(x)dm(y)\nonumber\\
&\le C J_{\phi^n}(x)m(R^n(x))).\label{equacaoJm}
\end{align}
Now take an arbitrary  $\ell\in\mathbb N$. 
Given any $x\in E_1\cap E_2\cap E_3$
choose $n(x)\ge\max\{k_1,k_2,\ell\}$ such that \eqref{equacao} holds. It follows from~\eqref{equacaoJ}, \eqref{equacao} and \eqref{equacaoJm} that
 $$ \frac{C}\eta m(R^{n(x)}(x)))\ge e^{-\beta n(x)},
 $$
 which together with~\eqref{muYn} gives
  \begin{equation}\label{sete}
  \mu(R^{n(x)}(x))\le e^{-n(x)(\alpha-\beta)} \frac{C}\eta m(R^{n(x)}(x))\le e^{-\ell(\alpha-\beta)} \frac{C}\eta m(R^{n(x)}(x)).
 \end{equation}
 Defining 
 $$\mathcal Q=\left\{ R_{n(x)}(x) : x\in E_1\cap E_2\cap E_3 \right\}$$
we have that $\mathcal Q$  is a $\mu$ mod 0 partition of $\Omega$. As two elements in $\mathcal Q$ are either disjoint or one contains the other, we can extract a $\mu$ mod 0 subcover $\widetilde{\mathcal Q}$ of $\Omega$ by pairwise disjoint sets.
 From \eqref{sete} we get
 $$1=\mu(\Omega)\le \sum_{Q\in\widetilde{\mathcal Q}}\mu(Q) \le e^{-\ell(\alpha-\beta)} \frac{C}\eta \sum_{Q\in\widetilde{\mathcal Q}}m(Q)
 \le e^{-\ell(\alpha-\beta)} \frac{C}\eta.
 $$
 Since $\alpha>\beta$ and $\ell$ can be taken arbitrarily large, this gives a contradiction. 
 
Assume now that $\mu$ is not an ergodic probability measure. By the Ergodic Decomposition Theorem, there exists a probability measure $\theta$   on the Borel sets of  $E_\phi(\Omega)$ such that for any measurable function $f:\Omega\to\mathbb R$ we have
 \begin{align*}
 \int_\Omega f d\mu &= \int_{E_\phi(\Omega)}\int_\Omega f d\nu d\theta(\nu),
 \end{align*}
where $E_\phi(\Omega)$  stands for the set of $\phi$-invariant ergodic probability measures    on the Borel sets of $\Omega$, endowed with the weak* topology. 
As we are assuming $\mathcal R$ quasi-Markovian with respect to $\mu$, it   follows  from the Ergodic Decomposition Theorem that $\mathcal R$ must be quasi-Markovian with respect to any  measure $\nu$ in the ergodic decomposition of~$\mu$. Moreover, by the concavity of the function $-x\log x$ and Jensen Inequality, we  have
$$H_\mu(\mathcal R)\ge  \int_{E_\phi(\Omega)} H_\nu(\mathcal R)   d\theta(\nu).
$$
As we are assuming $H_\mu(\mathcal R)<\infty$, we also have $H_\nu(\mathcal R) <\infty$ for $\theta$ almost every  measure $\nu$ in the ergodic decomposition of $\mu$.
Hence, by the case already seen for an ergodic measure,  we can write
 \begin{align*} 
h_\mu(\phi) &=\int_\Omega h_\mu(\mathcal R,x)d\mu\\
&=  \int_{E_\phi(\Omega)}\int_\Omega h_\mu(\mathcal R,x) d\nu d\theta(\nu)\\
&\le \int_{E_\phi(\Omega)}\int_\Omega\log J_\phi d\nu d\theta(\nu)\\
&=\int_\Omega\log J_\phi d\mu.
 \end{align*}
\subsection{Second inequality}
Let us now complete  the proof of   Theorem~\ref{pesin}.  Assume now that 
   $\mu$ is an absolutely continuous $\phi$-invariant probability measure   such that $H_\mu(\mathcal R_\phi)<\infty$ and~$\mathcal R_\phi$ is quasi-Markovian with respect to   $\mu$.
By Theorem~\ref{ruelle} it is enough to show that 
\begin{equation}\label{final}
h_{\mu}(\phi)\ge \int_\Omega \log(J_{\phi})d\mu.
\end{equation}
Consider  $\rho$ the density of $\mu$ with respect to $m$, given by Radon-Nikodym Theorem. We have for $\mu$ almost every $x\in \Omega$
\begin{equation}\label{Radon}
\lim_{n\to\infty}\frac{\mu(R^n(x))}{m(R^n(x))}=\rho(x)>0.
\end{equation}
Since 
$$\frac{1}{n} \log\mu(R^n(x))=\frac{1}{n} \log m(R^n(x))+\frac{1}{n} \log\frac{\mu(R^n(x))}{m(R^n(x))},
$$
it follows from~\eqref{Radon} that 
\begin{equation}\label{correcto}
h_{\mu}(\mathcal R, x)=\lim_{n\to\infty}-\frac{1}{n} \log \left(m (R^n(x) ) \right),
\end{equation}
for $\mu$ almost every $x\in \Omega$.
We may write
\begin{align}\label{numa}
m(\Omega)\ge m\left(\phi^n(R^n(x))\right)&=\int_{R^n(x)}J_{\phi^n}(y)dm(y)
=\int_{R^n(x)}\frac{J_{\phi^n}(y)}{J_{\phi^n}(x)}J_{\phi^n}(x)dm(y).
\end{align}
It follows from Lemma~\ref{le-distor} and  \eqref{numa} that
$$m(\Omega)\ge \frac1C J_{\phi^n}(x)m\left( R^n(x)\right)
$$
which then implies  
\begin{equation}\label{eq.last-1}
 -\frac1n\log m(R^n(x))\ge- \frac{\log C+\log m(\Omega)}n +\frac1n\log J_{\phi^n}(x).
\end{equation}
Using~\eqref{eq.entro},~\eqref{correcto} and~\eqref{eq.last-1} we easily deduce that 
\begin{equation*}
h_{\mu}(\phi)\ge\int\lim_{n\to\infty}\frac{1}{n} \log(J_{\phi^n}(x))d\mu=
\int \lim_{n\to\infty}\frac{1}{n} \sum_{i=0}^{n-1}\log(J_{\phi}(\phi^{i}(x)))d\mu.
\end{equation*}
Finally, using Birkhoff's Ergodic
Theorem we  obtain \eqref{final}.


\section{Quasi-Markovian property} \label{se.quasi-markovian}
The goal of this section is to prove Theorem~\ref{pesin2}. This will be obtained as a consequence of Theorem~\ref{pesin} and Proposition~\ref{pr.casiM} below, which 
gives a useful criterium for obtaining the quasi-Markovian property. 
We start by  recalling some general facts from \cite{ABV00} about maps (not necessarily piecewise expanding)  which are non-uniformly expanding   and have slow recurrence to a singular set.

Let $M$ be a compact manifold, $\mathcal S$ a compact subset of $M$
and $f:M\setminus\mathcal S\to M$ a $C^1$ map.  Typically, $\mathcal S$  is a  set of points where the map $f$ fails to be differentiable or even continuous, or even if $f$ is differentiable its derivative is not an isomorphism. 
For the sake of completeness, let us refer that in the  general setting of~\cite{ABV00}, where maps   with critical sets are allowed,  condition~(S1)  introduced  in  Subsection~\ref{se.entropf} above to express that $f$ behaves as a power of the distance close to $\mathcal S$   needs to be replaced by the following stronger condition:
 there are $B,\beta>0$ such that 
\begin{itemize}
\item[(S1$^*$)] 
$\displaystyle{
\frac{1}{B} \dist(x,\mathcal S)^{\beta}
\le \frac{\|Df(x)v\|}{\|v\|}
\le B \dist(x,\mathcal S)^{-\beta}}$,
%
%
\end{itemize}
for every $v\in T_x M$ and $x, y \in M\setminus\mathcal S$ 
with $\dist(x,y) < \dist(x,\mathcal S)/2$.
However, notice that for piecewise expanding maps the left hand side of (S1$^*$) is trivially satisfied, and so (S1$^*$) coincides with (S1) in this setting. Related to this, see Remark~\ref{re.s3} below.

Given $\delta>0$ and $x\in M\setminus\mathcal S$ we define
the {\em $\delta$-truncated distance\/} from $x$
to $\mathcal S$ as
\begin{equation*}
\label{e.truncate}
\dist_\delta(x,\mathcal S)= 
\begin{cases}   \dist(x,\mathcal S), & \text{if } \dist(x,\mathcal S)<\delta; \\
                 1, & \text{otherwise}.
\end{cases} 
\end{equation*}
We say that $f$ is \emph{non-uniformly expanding (NUE)} on a set $H\subset M$ if there is some $c>0$ such that for all  $x\in H$  
\begin{equation}
\label{e.lyapunov1}
\limsup_{n\to+\infty} \frac{1}{n}
    \sum_{j=0}^{n-1} \log \|Df(f^j(x))^{-1}\|<-c.
\end{equation}
Moreover,  we say that $f$ has  \emph{slow recurrence (SR)} to $\mathcal S$ on $H$  if for all $\varepsilon>0$ there exists
$\delta>0$ such that  for all $x\in H$ 
\begin{equation*}
\label{e.faraway1}
\limsup_{n\to+\infty} \frac{1}{n}
\sum_{j=0}^{n-1}-\log \dist_\delta(f^j(x),\mathcal S)
\le\varepsilon.
\end{equation*}
The next result is a  consequence of Lemmas 5.2 and 5.4 in~\cite{ABV00}, and will be used to prove a useful criterium for a partition of a piecewise expanding map to be quasi-Markovian.

\begin{lemma}\label{te.vx}
Let $f:M\setminus \mathcal S\to M$ be a $C^2$ map such that $f$ behaves as a power of the distance close to  $\mathcal S$. If NUE and SR hold for a set $H\subset M$, then  there exists $\delta_1>0$ such that for all $x\in H$ there are infinitely many $n\in \mathbb N$ and an open neighborhood $V_n(x)$ of $x$ which is mapped by
  $f^n$ diffeomorphically  onto the ball
of radius $\delta_1$ around~$f^n(x)$. 
\end{lemma}

 Though the results  in \cite{ABV00} are stated for boundaryless manifolds, it is not difficult to see that that the conclusion of Lemma~\ref{te.vx} still holds for a manifold $M$ with boundary, provided the boundary of $M$ is included in the singular set $\mathcal S$. 
 
 \begin{remark}\label{re.s3} In \cite{ABV00}, 
 a third condition (S3) is considered in the definition of a map behaving as a power of the distance close to the singular set. However, that condition (S3)  is only used in the proof of \cite[Corollary~5.3]{ABV00} to deduce a bounded distortion property on the sets~$V_n(x)$. Here we do not need that property.
 \end{remark}

 \subsection{Piecewise expanding maps}
 Let us now    go back to  piecewise expanding maps. As observed above, in this setting  we have condition (S1$^*$)   equivalent to  (S1).

 \begin{lemma}\label{distance}   Let   $ \phi :\Omega\to\Omega$ be a  $C^1$
piecewise expanding map      for which (S1)-(S2) hold. If $\log\dist(\cdot,\mathcal S_\phi)\in L^p(m)$ and $\mu$ is an ergodic absolutely continuous invariant probability measure for $\phi$ such that  $d\mu/dm\in L^q(m)$ with $1\le p,q\le\infty$ and $1/p+1/q=1$,  then $\phi$ has slow recurrence to $\mathcal S_\phi$ on a subset of $\Omega$ with full $\mu$ measure.
\end{lemma}
\begin{proof}
Define for $x\in \Omega\setminus \mathcal S_\phi$  
 $$\xi(x)=-\log\dist(x,\mathcal S_\phi).
 $$
Since we are assuming that $\log\dist(\cdot,\mathcal S_\phi)\in L^p(m)$ and   $d\mu/dm\in L^q(m)$ with $1\le p,q\le\infty$ and $1/p+1/q=1$,  then Hölder Inequality  gives  that  $\xi\in L^1(\mu)$. Hence,
 $$
 \lim_{\delta\to 0^+}\int_{\{\xi>-\log\delta\}}\xi d\mu_t =0.
 $$
Observing that    we have
 $$
 \chi_{\{\xi>-\log\delta\}} \xi =-\log\dist_\delta(\cdot,\mathcal S_\phi),
 $$
where $\chi_{\{\xi>-\log\delta\}}$ denotes  the characteristic function of the set ${\{\xi>-\log\delta\}}$,  it easily follows that for all $\varepsilon>0$ we can find $\delta>0$ such that
  $$\int -\log\dist_\delta(\cdot,\mathcal S_\phi)d\mu<\varepsilon. $$
Hence, using the ergodicity of $\mu$,  Birkhoff's Ergodic Theorem yields
\begin{equation*}
\label{e.faraway12}
\lim_{n\to+\infty} \frac{1}{n}
\sum_{j=0}^{n-1}-\log \dist_\delta(\phi_t^j(x),\mathcal S_\phi)=
\int -\log\dist_\delta(\cdot,\mathcal S_\phi)d\mu_t <\varepsilon
\end{equation*}
for $\mu$ almost every $x\in \Omega$.
\end{proof}

 \begin{proposition}\label{pr.casiM}
Let   $ \phi $ be a  $C^1$
piecewise expanding map      for which (S1)-(S2) hold. If $\log\dist(\cdot,\mathcal S_\phi)\in L^p(m)$ and $\mu$ is an absolutely continuous invariant probability measure for which  $d\mu/dm\in L^q(m)$ with $1\le p,q\le\infty$ and $1/p+1/q=1$,
   then  $\mathcal R_\phi$ is quasi-Markovian with respect to $\mu$.
 \end{proposition}

\begin{proof}
Considering, as before,  $R^n(x)$    the element in  $\mathcal R^n$ containing $x$, 
we need to show that there is  some constant  $\eta>0$ such that for $\mu$ almost every $x\in \Omega$,  there are infinitely many values of $n\in\mathbb N$ for which
\begin{equation*}
m(\phi^n(R^n(x)))\ge\eta.\end{equation*}
By Lemma~\ref{distance},  there is a set $H\subset T$ with full  $\mu$  measure such that
 $\phi$ has slow recurrence to $\mathcal S_\phi$ on $H$. On the other hand, since $\phi$ is a piecewise expanding map, then $\phi$ is clearly non-uniformly expanding on $H$. Hence, by Lemma~\ref{te.vx} there exists $\delta_1>0$ such that for all $x\in H$ there are infinitely many $n\in \mathbb N$ and an open neighborhood $V_n(x)$ of $x$ which is mapped by
  $f^n$ diffeomorphically  onto the ball
of radius $\delta_1$ around~$f^n(x)$. 
In such case, each set $V_n(x)$  is necessarily contained in $R^n(x)$, for  $V_n(x)$ is mapped by $\phi^n$  diffeomorphically onto its image. Recalling that this image is the ball
of radius $\delta_1>0$ around~$\phi^n(x)$,    we have
 $$m(\phi^n(R^n(x)))\ge m(\phi^n(V_n(x)))\ge \pi\delta_1^2,$$
thus having proved that $\mathcal R_\phi$ is quasi-Markovian with respect to $\mu$.
\end{proof}

\section{Continuity of entropy}\label{se.continuity}

In this section we prove Theorem~\ref{entropy}. This will be obtained as   a consequence of Theorem~\ref{pesin}, together with Lemma~\ref{AOT06} below. 
The proof of this lemma would be a straightforward  application of Hölder Inequality if the convergence   were in the norm of $L^p(m)$ and $1/p+1/q=1$. Nevertheless, slightly improving the regularity of functions,  we are   able to obtain a useful criterium for the case that convergence holds only  in the norm of $L^1(m)$.

\begin{lemma}\label{AOT06}
Consider $1< p,q\le \infty$ with $1/p+1/q<1$. Assume that
\begin{enumerate}
\item $(f_n)_n$   is a bounded sequence in $L^p(m)$ converging   to~$f\in~L^p(m)$  in~$L^1(m)$;
\item $(g_n)_n$ is a bounded sequence in $L^q(m)$.
\end{enumerate}
Then 
$$\lim_{n\to\infty}\int (f_n-f)g_ndm=0.$$
\end{lemma}
\begin{proof}
Take an arbitrary $\epsilon>0$. Consider   $M>0$ such that 
 $
 \|f\|_p\le M$,  $\|f_n\|_p\le M$ and $ \|g_n\|_q\le M
 $ for all $n\ge 1$, and
let $1\le r<\infty$ be such that $1/q+1/r=1$.
Define for each $n\ge 1$
 $$B_n=\left\{ x\in \Omega: |f_n(x)-f(x)|>\frac{\epsilon}{2Mm(\Omega)^{1/r}}\right\}.
 $$
 Since $\|f_n-f\|_1\to 0$ as $n\to\infty$, we necessarily have
 \begin{equation}\label{Bn}
\lim_{n\to\infty}m(B_n)=0.
\end{equation}
 We can write
 \begin{align}\label{intsuma}
\int_\Omega |f_n-f|\,|g_n| dm &= \int_{\Omega\setminus B_n}  |f_n-f|\,|g_n| dm +\int_{B_n} |f_n-f|\,|g_n| dm .
\end{align}
For the first term in the sum above we have
$$
 \int_{\Omega\setminus B_n}  |f_n-f|\,|g_n| dm \le \frac\epsilon{2Mm(\Omega)^{1/r}}\int_\Omega |g_n| dm\le  \frac\epsilon{2Mm(\Omega)^{1/r}}\|g_n\|_qm(\Omega)^{1/r}<\frac\epsilon2.
$$
Let us see that, for $n$ sufficiently large,  the second term in~\eqref{intsuma} can also be made smaller than $\epsilon/2$. Consider now $1\le s<\infty$ such that $1/p+1/s=1$. 
Notice that one necessarily has $s<q$, and so $g_n\in L^q(m)\subset L^s(m)$. Using  Hölder inequality we get
 \begin{equation}\label{eq.second}
\int_{B_n} |f_n-f|\,|g_n| dm\le \|f_n-f\|_p\|\chi_{B_n}g_n\|_s\le 2M \|\chi_{B_n}g_n\|_s.
\end{equation}
Moreover, taking $1\le u<\infty$ such that $s/q+1/u=1$
 $$
 \|\chi_{B_n}g_n\|_s^s=\int_{B_n}|g_n|^s dm\le \|\chi_{B_n}\|_u\|g_n^s\|_{q/s}\le m(B_n)^{1/u}\|g_n\|_q^s,
 $$
 which then gives 
\begin{equation}\label{eq.last}
 \|\chi_{B_n}g_n\|_s \le m(B_n)^{1/(su)}\|g_n\|_q\le M m(B_n)^{1/(su)}
\end{equation}
Hence, 
using~\eqref{Bn},~\eqref{intsuma}, \eqref{eq.second} and \eqref{eq.last}
we easily see that the second term on the right hand side of~\eqref{intsuma} can be made smaller that $\epsilon/2$ as well, for $n$ sufficiently large.
\end{proof}


Let us now complete the proof of  Theorem~\ref{entropy}. By Theorem~\ref{pesin}, it is enough to show that the function
$$I\ni t\longmapsto \int_\Omega \log(J_{\phi_t})d\mu_t$$
is continuous.
Let
$(t_n)_n$ be an arbitrary sequence in $I$ converging to $t_0\in I$. Considering 
$$\rho_n=\frac{d\mu_{t_n}}{dm},$$
we may write
\begin{align} \left|\int_\Omega \log(J_{\phi_{t_0}})d\mu_{t_0}\right.-
&
\left. \int_\Omega \log(J_{\phi_{t_n}})d\mu_{t_n} \right|\le\nonumber\\
  &\left|\int_\Omega (\log(J_{\phi_{t_0}})-\log(J_{\phi_{t_n}}))\rho_ndm \right|+\left|\int_\Omega(\rho_0-\rho_n) \log(J_{\phi_{t_0}})dm \right|.\label{eq.sum}
\end{align}
Using Lemma~\ref{AOT06} with $f_n=\log(J_{\phi_{t_n}})$ and $ g_n=\rho_n$ in the first term of the sum above, and $f_n=\rho_n$ and $ g_n=\log(J_{\phi_{t_0}})$ in the second one,
it immediately follows  that, under the assumptions of Theorem~\ref{entropy}, both terms in the 
sum converge to zero when $n$ goes to infinity, thus giving the desired conclusion.

%
%

\section{Maps with large branches}\label{se.large}
In this section we prove Theorem~\ref{entropy2}. Assume that    $(\phi_t)_{t\in I}$ is a family of $C^2$
piecewise expanding maps  for which     the assumptions of Theorem~\ref{entropy2} hold. In particular,  there is $q>d$ for which    $\log J_{\phi_t}\in L^q(m)$ and $\log J_{\phi_t}$ depends continuously on $t\in I$ in $L^1(m)$.
We are going to prove that, in  this setting,  the assumptions of Theorem~\ref{entropy} are verified. In practice, we only have to show that if we take $\rho_t=d\mu_t/dm$, then there is some $1<p\le\infty$ with $1/p+1/q<1$ such that   $\rho_t$ is uniformly bounded in $L^p(m)$ and $\rho_t$ depends continuously on $t\in I$ in $L^1(m)$.
We start with a preliminary result whose conclusion gives  the bounded distortion condition used in \cite[Theorem A]{APV17}.

\begin{lemma}\label{le.distortion}
Let $\phi$ be a $C^2$ piecewise expanding map.  If  the bounded distortion condition (P$_2$) holds, then for all  $R\in\mathcal R_\phi$ and all $x\in\inte\phi(R)$ we have
$$
\frac{\left\|D\left(J_\phi\circ\phi^{-1}_R\right)(x)\right\|}{\left|J_\phi\circ\phi^{-1}_R(x)\right|}\le \Delta,
$$
where $\Delta>0$ is the   constant   in (P$_2$).
\end{lemma}
\begin{proof}
It is enough to show that  for all $x\in\inte\phi(R)$  and all $1\le j\le d$  we have
$$
\left|\frac{\frac\partial{\partial x_j}
\left(J_\phi\circ\phi^{-1}_R\right)(x)}{J_\phi\circ\phi^{-1}_R(x)}\right|\le \Delta.
$$
In fact,
\begin{align}
&\left|\frac{\frac\partial{\partial x_j}
\left(J_\phi\circ\phi^{-1}_R\right)(x)}{J_\phi\circ\phi^{-1}_R (x)}\right|\nonumber
\\
&=
\left|\frac\partial{\partial x_j}\log (J_\phi\circ\phi^{-1}_R) (x)\right|\nonumber
\\
&=\left|\lim_{h\to0}\frac1h\left(\log J_\phi(\phi_R^{-1}(x_1,\dots,x_j+h,\dots,x_d))-\log J_\phi(\phi_R^{-1}(x_1,\dots,x_j,\dots,x_d))
\right)
\right|\nonumber
\\
&=\left|\lim_{h\to0}\frac1h\left(\log \frac{J_\phi(\phi_R^{-1}(x_1,\dots,x_j+h,\dots,x_d))}{J_\phi(\phi_R^{-1}(x_1,\dots,x_j,\dots,x_d))}
\right)
\right|\label{aga}
\end{align}
Now, by (P$_2$) we have
$$
\log \frac{J_\phi(\phi_R^{-1}(x_1,\dots,x_j+h,\dots,x_d))}{J_\phi(\phi_R^{-1}(x_1,\dots,x_j,\dots,x_d))}\le\Delta|h|
$$
This obviously implies that the expression in \eqref{aga} is bounded by $\Delta$.
\end{proof}

Now, using Lemma~\ref{le.distortion}, we easily deduce that, under conditions (1) and (2) of Theorem~\ref{entropy2}, 
the assumptions of \cite[Theorem A]{APV17} are verified. Hence, as we are assuming uniqueness of the absolutely continuous invariant probability measure,  using \cite[Theorem A]{APV17} we obtain  in this particular case that $\rho_t$ depends continuously on $t\in I$ in $L^1(m)$.

Let us finally  prove that there is some $1<p\le\infty$ with $1/p+1/q<1$ such that   $\rho_t$ is uniformly bounded in $L^p(m)$. Under  assumption (1) of Theorem~\ref{entropy2}, it follows   from \cite[Corollary 3.4]{APV17} that the  density $\rho_t$  of the absolutely continuous invariant probability measure $\mu_t$ is a function  with (uniformly)  bounded variation. Actually, we have 
 $$\var(\rho_t)\le K_1, \quad\text{with}\quad K_1=\sum_{n=0}^\infty \lambda^n K,$$
where $K$ is the constant in condition (2)  of Theorem~\ref{entropy2}.
By   Sobolev's Inequality
(see e.g. \cite[Theorem 1.28]{G84}) there is a constant $C>0$ (only
depending on the dimension~$d$) such that 
 \begin{equation}\label{bvp}
\|\rho_t\|_p\leq C\var(\rho_t),
\quad \mbox{with}\quad p=\frac{d}{d-1}.
\end{equation}
This gives that $\rho_t$ is uniformly bounded in $L^{p}(m)$. Finally, observing that
 $$\frac1p+\frac1q=\frac{d-1}d+\frac1q<\frac{d-1}d+\frac1d=1,
 $$
 we conclude  the proof of Theorem~\ref{entropy2}.

\section{Tent maps} 
In this section we prove Theorems~\ref{main2} and~\ref{main1}. Both proofs are  based in the simple fact that if, for   $t$ belonging to some set of parameters $I$, 
the map $T_t$ has a unique absolutely continuous invariant probability measure $\mu_t$ which also happens to be the unique absolutely continuous invariant probability measure for a certain power $T_t^n$, and the family  $\{T_t^n\}_{t\in I}$ is in the conditions of Corollary~\ref{pesin3} and Theorem~\ref{entropy2}, then  
the entropy formula still holds for $T_t$, and $h_{\mu_t} (\tent_t)$ depends continuously on $t\in I$.
This  is actually  a straightforward  consequence of  the following well-known formulas establishing  that for each $t\in I$ one has
 $$h_{\mu_t} (\tent_t)=\frac1nh_{\mu_t} (\tent_t^n)$$
 and 
 $$
 \int \log J_{\tent_t^n} d\mu_t=  n\int \log J_{\tent_t} d\mu_t,
 $$

 \subsection{Two-dimensional tents}
 In this section we  obtain  Theorem~\ref{main2}  as a consequence of 
 Theorem~\ref{entropy2}.   As assumption (1) in  Theorem~\ref{entropy2} does not hold for $\tent_t$,  we need to take some iterate greater than one. This has already been considered in \cite{APV17} for obtaining the statistical stability, and $\tent_t^6$ is actually enough.  
 As proved in \cite[Theorem 1.1]{PRT15}, each $\tent_t$ is \emph{strongly transitive}, meaning that every non-empty open set becomes the whole attractor under a finite number of iterations by~$\tent_t$.  This in particular implies that the  absolutely continuous invariant probability measure  $\mu_t$ for $\tent_t$ must be  unique and ergodic and the strongly transitive attractor of $\tent_t$ mentioned above coincides with the support of $\mu_t$. Moreover, for any $t\in [\tau,1],$ any power of $\tent_t$ is also strongly transitive in the support of $\mu_t$ from which we deduce that any power of $\tent_t$ has a unique ergodic absolutely continuous invariant probability measure as well, which must necessarily coincide with~$\mu_t$. All these facts can be checked in \cite{PRT15}.



\begin{figure}[!ht]
\begin{minipage}{1 \linewidth}
\centering
\subfigure[]{
\centering
\includegraphics[width=0.4\textwidth]{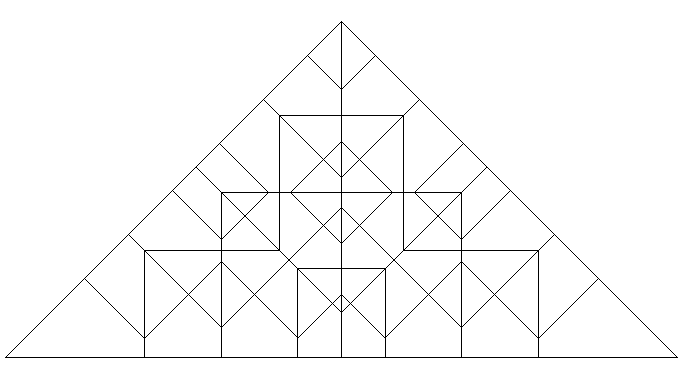}
}
\hspace{-1mm}
\subfigure[]{
\centering
\includegraphics[width=0.4\textwidth]{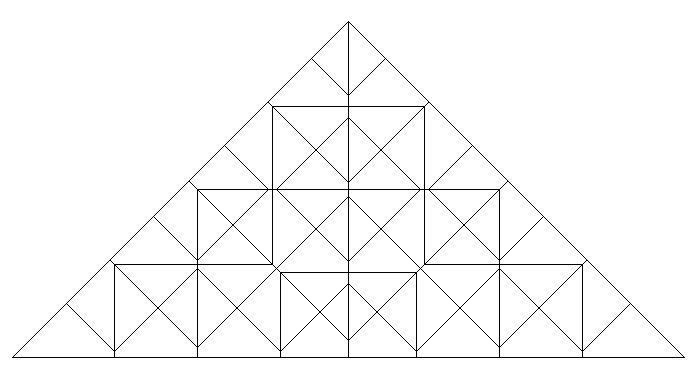}
}
\hspace{-1mm}
\subfigure[]{
\centering
\includegraphics[width=0.4\textwidth]{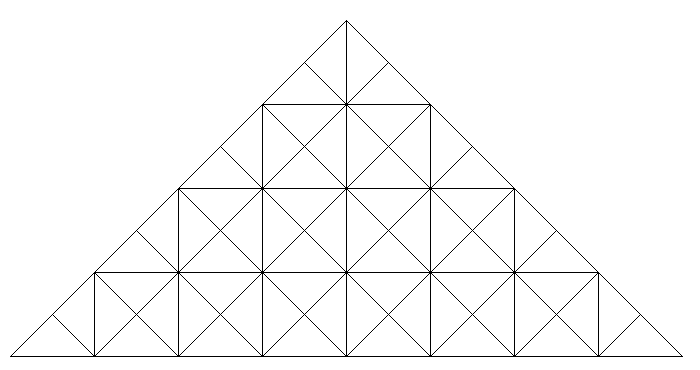}
}
\caption{\em Domains of smoothness for $ \tent_t^6 $: (a) $ t=\tau $ (b) $ t=0.95 $ (c) $ t=1 $}
\label{tres}
\end{minipage}
\end{figure}

We are going to see that the family $(\tent_t^6)_{t\in[\tau,1]}$ is in the conditions of Corollary~\ref{pesin3} and Theorem~\ref{entropy2}.
Observe that as for each $t\in[\tau,1]$ the partition into smoothness domains of  $\tent_t^6$ is finite, then it necessarily has  finite $\mu_t$ entropy.
Moreover, from  \cite[Section 4]{APV17} we know that each $\tent_t^6:\mathcal{T}\to \mathcal{T}$ is a piecewise expanding map with bounded distortion and large branches satisfying  conditions~(1) and~(2) in the statement of Theorem~\ref{entropy2}.  
Additionally, since  each map $\tent_t^6$ is piecewise linear with the slopes and smoothness domains  continuously depending on  the parameter $t$ (see Figure \ref{tres}), it is not difficult to see  that  $ \log(J_{\tent_t^6})\in L^\infty(m)$ and, moreover,  $ \log(J_{\tent_t^6})$ depends continuously  on  $t\in[\tau,1]$ in  the norm of $ L^1(m)$. 
This is the content of  condition (3) with $q=\infty$ in the statement  of Theorem~\ref{entropy2}. 

It remains to check that the assumptions of Corollary~\ref{pesin3} hold for our family of tent maps.
First of all, observe that condition~\eqref{betasigma} is part of condition (1) in~Theorem~\ref{entropy2}, and so it is satisfied. 
Considering $\mathcal S_t$ the singular set of $\tent_t$, in this case we have
\begin{equation*}\label{criticalextended}
\mathcal S_t=\mathcal C_6\cup \partial \mathcal{T},
\end{equation*}
where $\mathcal C_6$ is the critical set for $ \tent_t^6$ and $\partial \mathcal{T}$ is the boundary of $\mathcal{T}$. Notice that $\mathcal C_6$ is made by a finite number of  straight line  segments dividing $\mathcal{T}$ into the   smoothness domains of~$\tent_t^6$; see Figure~\ref{tres}. Since $\tent^6_t$ has constant derivative on each connected component of $\mathcal{T}\setminus \mathcal S_t$, then conditions  (S1)-(S2) are obviously satisfied. 
%
%
Finally, as $\mathcal S_t$ is a finite union of one-dimensional submanifolds of $\mathcal{T}$, it follows from \cite[Proposition~4.1]{AA04} that $\log\dist(\cdot,\mathcal S_t)\in L^2(m)$. Hence, all the assumptions of Corollary~\ref{pesin3} hold for the two-dimensional tent maps.

\subsection{One-dimensional tents} 
Here we  prove  Theorem~\ref{main1}.
Let us recall the family of one-dimensional tent maps $ \{T_t\}_{t \in (1,2]} ,$ with
$T_t:[0,2]\rightarrow [0,2]$ given by
\begin{equation*}
T_t(x)= \left\{
\begin{array}{ll}
tx, & \mbox{if }    0\le x\le 1;\\
t(2-x), & \mbox{if } 1\le x\le 2.%
\end{array}
\right.
\end{equation*}

Our strategy to prove Theorem~\ref{main1} is, once again, to use Corollary~\ref{pesin3} and Theorem~\ref{entropy2}.

Let us start by assuming that $t\in (2^{\frac{1}{4}},2].$ In this case, it is easy to see that, if $t\in (2^{\frac{1}{2}},2],$ then the interval $A_{t}=[t(2-t),t] $ is the unique attractor for $T_t$ and, moreover, that $T_t$ is strongly transitive on $A_t.$ On the other hand, if $t\in (2^{\frac{1}{4}},2^{\frac{1}{2}}],$ then $T_t$ has a unique attractor, still denoted by $A_t,$ formed by two disjoint subintervals (with a common endpoint, in the case $t=2^{\frac{1}{2}}$); see \cite{MS93} or \cite{M11} for details. In fact, we have $A_t=A_t^1\cup A_t^2$ with $T_t(A_t^1)=A_t^2,$ and
$T_t(A_t^2)=A_t^1.$ Moreover, the map $T_t^2$ is strongly transitive on $A_t^1$ and on $A_t^2.$ In any case, the map $T_t^5$ is transitive on $A_t$ for every $t\in (2^{\frac{1}{4}},2]$ (observe that, for instance, this last claim is not true for $T_t^6$). 
For each $t\in (1,2],$ let $\mu_t$ be the unique ergodic absolutely continuous invariant measure for $T_t$; see~\cite{LY73}. Following the same arguments given for the two-dimensional case we  may assert that, for every $t\in (2^{\frac{1}{4}},2],$ we have $\mu_t$ as the unique ergodic absolutely continuous invariante measure for $T_t^5.$

 Let us now explain why we choose the fifth power of the maps for parameters in the interval $(2^{\frac{1}{4}},2]$.
Observe that for each $t\in(1,2]$ we have $|T_t'|=t$, which then implies that the piecewise expanding condition (P$_1$) in Subsection~\ref{se.pebd} is satisfied by $T_t$ with $\sigma_t=1/t$. 
However, for having  condition~\eqref{betasigma} satisfied we need to take powers of $T_t$.
Recall that in the one-dimensional case we can always assume $\beta_t=1$; see Remark~\ref{re.beta}. Moreover, for any $k\in \mathbb N$,  the map $T_t^k$ satisfies condition (P$_1$) with $\sigma_t^k=1/t^k$.
Now, a straightforward calculation gives that \eqref{betasigma} holds  for $T_t^5$, uniformly in $t\in (2^{\frac{1}{4}},2].$

Now, observe that the singular set $\mathcal S_t$ of $ T_t^5 $ is formed by a finite number of critical points where the map is not differentiable,  together with the boundary points $0$ and $2$. As $T_t^5$ has constant derivative on each connected component of $[0,2]\setminus \mathcal S_t,$ conditions (S1)-(S2) are trivially satisfied. Finally, it is easy to see that in this one-dimensional setting we have $\log\dist(\cdot,\mathcal S_t)\in L^1(m)$. Hence, all the assumptions of Corollary~\ref{pesin3} are satisfied for the maps~$  T_t^5$, with $t \in (2^{\frac{1}{4}},2] .$ In this way we deduce that the entropy formula holds for $\mu_t,$ whenever $t\in (2^{\frac{1}{4}},2].$

Next,  observe that since the maps $ T_t^5 $ are continuous, then the second condition in Theorem~\ref{entropy2} is trivially satisfied. This fact together with condition~\eqref{betasigma} allow us to assert that $\mu_t$ depends continuously on $t\in (2^{\frac{1}{4}},2]$ according to  \cite[Theorem A]{APV17}.
Furthermore, it is obvious that  $ \log(J_{T_t^5})\in L^\infty(m)$ and, moreover,  $ \log(J_{T_t^5})$ depends continuously  on  $t\in (2^{\frac{1}{4}},2]$ in  the norm of $ L^1(m)$. 
This is the content of  condition (3) in the statement  of Theorem~\ref{entropy2} with $q=\infty$. Hence, Theorem~\ref{main1} is proved for parameters $t\in (2^{\frac{1}{4}},2].$

Now we explain how we can extend these ideas to the whole interval of parameters $(1,2].$ For this we write
$$
(1,2]=\bigcup_{j=0}^{\infty}{\mathcal{I}_j},\quad\text{where} \quad\mathcal{I}_j=(2^{\frac{1}{2^{j+2}}},2^{\frac{1}{2^j}}],
$$ 
and and prove Theorem~\ref{main1} for every $t\in \mathcal{I}_j$ and  every $j\in \mathbb{N}.$ Observe that, for every $j\ge0$ we have $\mathcal{I}_j \cap \mathcal{I}_{j+1}=(2^{\frac{1}{2^{j+2}}},2^{\frac{1}{2^{j+1}}}]$ and therefore the continuity of the entropy on the sequence of parameters $\{\frac{1}{2^j}\}_{j\in \mathbb{N}}$ will also be guaranteed.

Let us briefly describe the dynamics of $T_t$ for parameters $t\in \mathcal{I}_j.$ If $t\in \mathcal{I}_j$ is such that $t\leq 2^{\frac{1}{2^{j+1}}}$ then $T_t$ has an attractor $A_t$ formed by $2^{j+1}$ disjoint pieces and $T_t^{2^{j+1}}$ is strongly transitive on any of these pieces; see \cite{MS93} or \cite{M11} for details. If $t\in \mathcal{I}_j$ is such that $t> 2^{\frac{1}{2^{j+1}}}$ then $T_t$ has an attractor $A_t$ formed by $2^{j}$ disjoint pieces and $T_t^{2^j}$ is strongly transitive on any of these pieces. In any case, $T^{2^{j+2}+1}$ is transitive on $A_t$ for every $t\in \mathcal{I}_j$ and, moreover, it is easy to see that condition \eqref{betasigma} holds  for $T^{2^{j+2}+1}$, uniformly in $\mathcal{I}_j.$ The rest of the arguments needed for proving Theorem~\ref{main1} for every $t\in \mathcal{I}_j$ follows in the same way as the ones used before for $t\in \mathcal{I}_0=(2^{\frac{1}{4}},2].$


\begin{thebibliography}{10}

\bibitem{A00}
J.~F. Alves.
\newblock S{RB} measures for non-hyperbolic systems with multidimensional
  expansion.
\newblock {\em Ann. Sci. \'Ecole Norm. Sup. (4)}, 33(1):1--32, 2000.

\bibitem{AA04}
J.~F. Alves and V.~Ara{\'u}jo.
\newblock Hyperbolic times: frequency versus integrability.
\newblock {\em Ergodic Theory Dynam. Systems}, 24(2):329--346, 2004.

\bibitem{ABV00}
J.~F. Alves, C.~Bonatti, and M.~Viana.
\newblock S{RB} measures for partially hyperbolic systems whose central
  direction is mostly expanding.
\newblock {\em Invent. Math.}, 140(2):351--398, 2000.

\bibitem{ACF10}
J.~F. Alves, M.~Carvalho, and J.~M. Freitas.
\newblock Statistical stability and continuity of {SRB} entropy for systems
  with {G}ibbs-{M}arkov structures.
\newblock {\em Comm. Math. Phys.}, 296(3):739--767, 2010.

\bibitem{ACF10a}
J.~F. Alves, M.~Carvalho, and J.~M. Freitas.
\newblock Statistical stability for {H}{\'e}non maps of the
  {B}enedicks-{C}arleson type.
\newblock {\em Ann. Inst. H. Poincar\'e Anal. Non Lin\'eaire}, 27(2):595--637,
  2010.

\bibitem{AOT06}
J.~F. Alves, K.~Oliveira, and A.~Tahzibi.
\newblock On the continuity of the {SRB} entropy for endomorphisms.
\newblock {\em J. Stat. Phys.}, 123(4):763--785, 2006.

\bibitem{APV17}
J.~F. Alves, A.~Pumari{\~n}o, and E.~Vigil.
\newblock Statistical stability for multidimensional piecewise expanding maps.
\newblock {\em Proc. Amer. Math. Soc.}, 145(7):3057--3068, 2017.

\bibitem{AS14}
J.~F. Alves and M.~Soufi.
\newblock Statistical stability of geometric {L}orenz attractors.
\newblock {\em Fund. Math.}, 224(3):219--231, 2014.

\bibitem{APP09}
V.~Ara{\'u}jo, M.~J. Pac{\'\i}fico, E.~R. Pujals, and M.~Viana.
\newblock Singular-hyperbolic attractors are chaotic.
\newblock {\em Trans. Amer. Math. Soc.}, 361(5):2431--2485, 2009.

\bibitem{BJ12}
A.~Barrio~Blaya and V.~Jim{\'e}nez~L{\'o}pez.
\newblock On the relations between positive {L}yapunov exponents, positive
  entropy, and sensitivity for interval maps.
\newblock {\em Discrete Contin. Dyn. Syst.}, 32(2):433--466, 2012.

\bibitem{BC85}
M.~Benedicks and L.~Carleson.
\newblock On iterations of {$1-ax^2$} on {$(-1,1)$}.
\newblock {\em Ann. of Math. (2)}, 122(1):1--25, 1985.

\bibitem{BC91}
M.~Benedicks and L.~Carleson.
\newblock The dynamics of the {H}{\'e}non map.
\newblock {\em Ann. of Math. (2)}, 133(1):73--169, 1991.

\bibitem{BY92}
M.~Benedicks and L.-S. Young.
\newblock Absolutely continuous invariant measures and random perturbations for
  certain one-dimensional maps.
\newblock {\em Ergodic Theory Dynam. Systems}, 12(1):13--37, 1992.

\bibitem{BY93a}
M.~Benedicks and L.-S. Young.
\newblock Sinai-{B}owen-{R}uelle measures for certain {H}{\'e}non maps.
\newblock {\em Invent. Math.}, 112(3):541--576, 1993.

\bibitem{M11}
W.~de~Melo.
\newblock Renormalization in one-dimensional dynamics.
\newblock {\em J. Difference Equ. Appl.}, 17(8):1185--1197, 2011.

\bibitem{MS93}
W.~de~Melo and S.~van Strien.
\newblock {\em One-dimensional dynamics}, volume~25 of {\em Ergebnisse der
  Mathematik und ihrer Grenzgebiete (3) [Results in Mathematics and Related
  Areas (3)]}.
\newblock Springer-Verlag, Berlin, 1993.

\bibitem{DKU90}
M.~Denker, G.~Keller, and M.~Urba{\'n}ski.
\newblock On the uniqueness of equilibrium states for piecewise monotone
  mappings.
\newblock {\em Studia Math.}, 97(1):27--36, 1990.

\bibitem{D14}
N.~Dobbs.
\newblock On cusps and flat tops.
\newblock {\em Ann. Inst. Fourier (Grenoble)}, 64(2):571--605, 2014.

\bibitem{F05}
J.~M. Freitas.
\newblock Continuity of {SRB} measure and entropy for {B}enedicks-{C}arleson
  quadratic maps.
\newblock {\em Nonlinearity}, 18(2):831--854, 2005.

\bibitem{G84}
E.~Giusti.
\newblock {\em Minimal surfaces and functions of bounded variation}, volume~80
  of {\em Monographs in Mathematics}.
\newblock Birkh\"auser Verlag, Basel, 1984.

\bibitem{GB89}
P.~G{{\'o}}ra and A.~Boyarsky.
\newblock Absolutely continuous invariant measures for piecewise expanding
  {$C^2$} transformation in {${\bf R}^N$}.
\newblock {\em Israel J. Math.}, 67(3):272--286, 1989.

\bibitem{H91}
F.~Hofbauer.
\newblock An inequality for the {L}japunov exponent of an ergodic invariant
  measure for a piecewise monotonic map of the interval.
\newblock In {\em Lyapunov exponents ({O}berwolfach, 1990)}, volume 1486 of
  {\em Lecture Notes in Math.}, pages 227--231. Springer, Berlin, 1991.

\bibitem{KSL86}
A.~Katok, J.-M. Strelcyn, F.~Ledrappier, and F.~Przytycki.
\newblock {\em Invariant manifolds, entropy and billiards; smooth maps with
  singularities}, volume 1222 of {\em Lecture Notes in Mathematics}.
\newblock Springer-Verlag, Berlin, 1986.

\bibitem{K89}
G.~Keller.
\newblock Lifting measures to {M}arkov extensions.
\newblock {\em Monatsh. Math.}, 108(2-3):183--200, 1989.

\bibitem{LY73}
A.~Lasota and J.~A. Yorke.
\newblock On the existence of invariant measures for piecewise monotonic
  transformations.
\newblock {\em Trans. Amer. Math. Soc.}, 186:481--488, 1973.

\bibitem{L81}
F.~Ledrappier.
\newblock Some properties of absolutely continuous invariant measures on an
  interval.
\newblock {\em Ergodic Theory Dynamical Systems}, 1(1):77--93, 1981.

\bibitem{L84}
F.~Ledrappier.
\newblock Propri{\'e}t{\'e}s ergodiques des mesures de {S}inai.
\newblock {\em Inst. Hautes {\'E}tudes Sci. Publ. Math.}, (59):163--188, 1984.

\bibitem{LS82}
F.~Ledrappier and J.-M. Strelcyn.
\newblock A proof of the estimation from below in {P}esin's entropy formula.
\newblock {\em Ergodic Theory Dynam. Systems}, 2(2):203--219 (1983), 1982.

\bibitem{LY85}
F.~Ledrappier and L.-S. Young.
\newblock The metric entropy of diffeomorphisms. {I}. {C}haracterization of
  measures satisfying {P}esin's entropy formula.
\newblock {\em Ann. of Math. (2)}, 122(3):509--539, 1985.

\bibitem{P77}
Y.~B. Pesin.
\newblock Characteristic {L}japunov exponents, and smooth ergodic theory.
\newblock {\em Uspehi Mat. Nauk}, 32(4 (196)):55--112, 287, 1977.

\bibitem{PRT14}
A.~Pumari{\~n}o, J.~A. Rodr{\'\i}guez, J.~C. Tatjer, and E.~Vigil.
\newblock Expanding {B}aker maps as models for the dynamics emerging from
  3{D}-homoclinic bifurcations.
\newblock {\em Discrete Contin. Dyn. Syst. Ser. B}, 19(2):523--541, 2014.

\bibitem{PRT15}
A.~Pumari{\~n}o, J.~A. Rodr{\'\i}guez, J.~C. Tatjer, and E.~Vigil.
\newblock Chaotic dynamics for two-dimensional tent maps.
\newblock {\em Nonlinearity}, 28:407--434, 2015.

\bibitem{PT06}
A.~Pumari{\~n}o and J.~C. Tatjer.
\newblock Dynamics near homoclinic bifurcations of three-dimensional
  dissipative diffeomorphisms.
\newblock {\em Nonlinearity}, 19(12):2833--2852, 2006.

\bibitem{QXZ09}
M.~Qian, J.-S. Xie, and S.~Zhu.
\newblock {\em Smooth ergodic theory for endomorphisms}, volume 1978 of {\em
  Lecture Notes in Mathematics}.
\newblock Springer-Verlag, Berlin, 2009.

\bibitem{R78}
D.~Ruelle.
\newblock An inequality for the entropy of differentiable maps.
\newblock {\em Bol. Soc. Brasil. Mat.}, 9(1):83--87, 1978.

\bibitem{T01a}
J.~C. Tatjer.
\newblock Three-dimensional dissipative diffeomorphisms with homoclinic
  tangencies.
\newblock {\em Ergodic Theory Dynam. Systems}, 21(1):249--302, 2001.

\end{thebibliography}

\end{document}